\documentclass[11pt]{article}
\usepackage{color}
\definecolor{aleacolour}{rgb}{0.09,0.32,0.44} 
\usepackage[pdfborder={0 0 0},pdfborderstyle={},colorlinks=true,linkcolor=aleacolour,citecolor=aleacolour,urlcolor=blue]{hyperref}
\usepackage{amsthm,amsmath,amssymb,amsfonts}
\usepackage{enumerate}
\usepackage{cases}
\usepackage[top=1in,bottom=1in,left=1in,right=1in]{geometry}

\def\({\left(}
\def\){\right)}

\newcommand{\cE}{\mathcal{E}}

\newcommand{\sstar}{{s^*}}
\newcommand{\fraction}[1]{\{ #1 \}} 
\newcommand{\integer}[1]{\lfloor #1 \rfloor} 
\newcommand{\ceiling}[1]{\lceil #1 \rceil}

\newcommand{\sset}{\mathcal{AP}}
\newcommand{\fv}{{f^V}}
\newcommand{\fvv}[1]{{f_{#1}^V}}

\theoremstyle{plain}
\newtheorem{theorem}{Theorem}

\newtheorem{lemma}[theorem]{Lemma}
\newtheorem{claim}[theorem]{Claim}
\newtheorem{corollary}[theorem]{Corollary}
\newtheorem{conjecture}[theorem]{Conjecture}

\theoremstyle{definition}

\theoremstyle{remark}

\theoremstyle{property}

\title{Short proof of the asymptotic confirmation of the Faudree-Lehel Conjecture}
\author{
Jakub Przyby{\l}o\thanks{AGH University of Science and Technology, Faculty of Applied Mathematics, al. A. Mickiewicza 30, 30-059 Krakow, Poland. 
 Email: \href{mailto:jakubprz@agh.edu.pl} {\nolinkurl{jakubprz@agh.edu.pl}}.}
\and
Fan Wei\thanks{Department of Mathematics, Princeton University, Princeton, NJ 08544. Research supported by NSF Award DMS-1953958. Email: \href{mailto:fanw@princeton.edu} {\nolinkurl{fanw@princeton.edu}}.}
}
\date{}

\begin{document}
\maketitle  

\begin{abstract}
Given a simple graph $G$, the {\it irregularity strength} of $G$, denoted $s(G)$, is the least positive integer $k$ such that there is a weight assignment on edges $f: E(G) \to \{1,2,\dots, k\}$ 
for which each vertex weight $\fv(v):= \sum_{u: \{u,v\}\in E(G)} f(\{u,v\})$ is unique amongst all $v\in V(G)$.
In 1987, Faudree and Lehel conjectured that there is a constant $c$ such that $s(G) \leq n/d + c$ for all $d$-regular graphs $G$ on $n$ vertices with $d>1$, whereas it is trivial that $s(G) \geq n/d$. 
In this short note we prove that 
the Faudree-Lehel Conjecture holds when $d \geq n^{0.8+\epsilon}$ for any fixed $\epsilon >0$, with a small additive constant $c=28$ for $d$ large enough. Furthermore, we confirm the conjecture asymptotically by proving that for any fixed $\beta\in(0,1/4)$ there is a constant $C$  such that for all $d$-regular graphs $G$, $s(G) \leq \frac{n}{d}(1+\frac{C}{d^{\beta}})+28$, extending and improving a recent result of Przyby{\l}o that 
$s(G) \leq \frac{n}{d}(1+ \frac{1}{\ln^{\epsilon/19}n})$ 
whenever $d\in [\ln^{1+\epsilon} n, n/\ln^{\epsilon}n]$ 
and $d$ is large enough.
\end{abstract}

\section{Introduction}
Let $G$ be a simple graph with $n$ vertices. For a positive integer $k$ an edge-weighting function $f: E(G)\to \{1,2,\dots, k\}$ is called \emph{$k$-irregular} if the weighted degrees, denoted by $\fv(v) = \sum_{u\in N(v)} f(\{v,u\})$ are distinct for $v \in V(G)$; we will call $f(\{u,v\})$ and $\fv(v)$ simply the \emph{weights} of $\{u,v\}$ and $v$. The {\it irregularity strength} of $G$, denoted $s(G)$, is the least $k$, if exists, for which 
there is such a $k$-irregular edge-weighting function $f$; we set $s(G)=\infty$ otherwise.  It is easy to see that $s(G) < \infty$ if and only if $G$ has no isolated edges and at most one isolated vertex~\cite{FGKP}.

The irregularity strength was first introduced by Chartrand, 
Jacobson, Lehel,  Oellermann,  Ruiz, and Saba~\cite{CJLORS}. 
Later an optimal general bound $s(G) \leq n -1$ 
was proved in~\cite{AT,Nsingle}  for all graphs with finite irregularity strength  except for $K_3$. 
This bound occurred to be far from optimum for graphs with larger minimum degree. 
Special concern was in this context  devoted to $d$-regular graphs.
In~\cite{FL}
 Faudree and Lehel 
 showed $s(G) \leq \ceiling{n/2}+9$ for these. 
By a simple counting argument, it is easy to see that on the other hand, 
\[s(G) \geq \ceiling{(n+d+1)/d}.\] This lower bound motivated Faudree and Lehel to conjecture that $(n/d)$ is close to optimal, as proposed in~\cite{FL} in 1987.  In fact this conjecture was first posed by Jacobson, as mentioned in~\cite{L-survey}. 
\begin{conjecture}[\cite{FL}]\label{conj:main}
There is a constant $C>0$ such that for all $d$-regular graphs $G$ on $n$ vertices and with $d>1$, $s(G) \leq \frac{n}{d}+C$. 
\end{conjecture}
It is this conjecture that ``energized the study of the irregularity strength'', as stated in~\cite{CL}, and many related subjects throughout the following decades. It remains open 
 after more than thirty years since its formulation. A significant step forward towards solving it was 
achieved in 2002 by Frieze, Gould, Karo\'nski, and Pfender, who used the probabilistic method to prove the first linear bound $s(G) \leq 48(n/d)+1$ for $d \leq \sqrt{n}$, and a super-linear one $s(G) \leq 240(\log n)(n/d)+1$ in the remaining cases. 
The linear bound in $n/d$ 
was further extended to the case when $d \geq 10^{4/3}n^{2/3}\log^{1/3}n$ 
by  Cuckler and Lazebnik~\cite{CL}. The first general and unified linear bound in 
$n/d$ 
for the full spectrum of
$(n,d)$ was delivered by Przyby{\l}o~\cite{Plinear1, Plinear2}, who used a constructive rather than random 
approach to prove the bound $s(G) \leq 16(n/d)+6$.
Since then several works based on inventive new algorithms have been conducted to improve the multiplicative constant in front of $n/d$, see e.g.~\cite{Kthesis, KKP, MP}. 
The best result among these for any value of $d$ is due to 
Kalkowski, Karo{\'n}ski, and Pfender~\cite{KKP}, who showed that in general $s(G) \leq 6 \ceiling{n/\delta}$ for graphs with minimum degree $\delta\geq 1$ and without isolated edges. 
Only just recently it was proved by  Przyby{\l}o~\cite{Pasy}
that the Faudree-Lehel Conjecture holds asymptotically almost surely for random graphs $G(n,p)$ (which are typically ``close to'' regular graphs), for any constant $p$, and holds asymptotically (in terms of $d$ and $n$)  for $d$ not in extreme values. 
\begin{theorem}[Przyby{\l}o~\cite{Pasy}]\label{thm:previous}
Given any $\epsilon>0$, for every $d$-regular graph $G$ with $n$ vertices and $d\in [\ln^{1+\epsilon} n, n/\ln^{\epsilon}n]$, if $n$ is sufficiently large,
\[
s(G) \leq \frac{n}{d}\left(1+ \frac{1}{\ln^{\epsilon/19}n}\right).
\]
\end{theorem} 

In \cite{Pasy}, Przyby{\l}o moreover mentioned that ``a poly-logarithmic in $n$ lower bound on $d$ is unfortunately unavoidable" within his approach. 
In this paper we present an argument which is firstly quite short, secondly 
bypasses the mentioned poly-logarithmic in $n$ lower bound and extends the asymptotic bound to 
all possible cases
$1\leq d \leq n-1$ 
and thirdly, 
the upper bound we present is stronger than the one in Theorem~\ref{thm:previous} (where in particular $\ln^{\epsilon/19}n \ll \ln^{(1+\epsilon)/19} n \leq d^{1/19}$).

\begin{theorem}\label{thm:main2}
Given any $0< \beta< 1/4$, for every $d$-regular graph $G$ with $n$ vertices, if $d$ is sufficiently large in terms of $\beta$,
\[s(G) < \frac{n}{d}\left(1+ \frac{14}{d^{\beta}}\right)+28.\]
\end{theorem}
%
\begin{corollary}\label{cor:main2}
Given any $0< \beta< 1/4$, there is a constant $C$ such that 
for every $d$-regular graph $G$ with $n$ vertices, $s(G) < \frac{n}{d}\left(1+ \frac{C}{d^{\beta}}\right)+28$.
\end{corollary}

The second contribution of this paper is a confirmation 
that the 
Faudree-Lehel Conjecture, i.e. Conjecture \ref{conj:main},
holds literally (not only asymptotically) for ``dense" graphs, i.e., whenever $d\geq n^{0.8+\epsilon}$
for any fixed $\epsilon>0$.

\begin{theorem}\label{thm:main1}
Given any $0< \beta < 1/4$, for every $d$-regular graph $G$ on $n$ vertices with $d^{1+\beta} \geq n$, if $d$ is sufficiently large in terms of $\beta$, then  \[s(G) < {n/d} + 28.\]
\end{theorem}

\begin{corollary}\label{cor:main1}
Given any $0< \beta < 1/4$, there is a constant $C$ such that for every $d$-regular graph $G$ on $n$ vertices with $d^{1 + \beta} \geq n$, $s(G) < {n/d} + C$. 
\end{corollary}


We remark that similar conclusions as the ones above can also be derived from~\cite{FWJP1}, 
which describes on almost 30 pages a very long, multistage and technically complex random construction yielding general results for all graphs (not only regular graphs). Taking into account that Conjecture \ref{conj:main} remains 
a central open question  
of the related field, 
cf.~\cite{FWJP1} for more comprehensive exposition of the history and relevance of this problem, 
we decided to present separately this very concise argument concerning the conjecture itself, 
which is also dramatically 
easier to follow. 
Moreover, the present proof is a local lemma based argument, and thus is very  different from the one in~\cite{FWJP1}, which might also be beneficial for further research. Another note is that the current proof is somewhat using the full power of the local lemma, where the symmetric version of local lemma is insufficient. 
%
%
%
Lastly, unlike in~\cite{FWJP1}, we also provide a specific additive constant in the obtained bounds for regular graphs, in particular in Theorem~\ref{thm:main1}, which is relatively small.

\section{Proof of main results} 
\subsection{Preliminaries}
For a set $U \subset V(G)$ and a vertex $v \in V(G)$, we use $\deg_U(v)$ to denote the number of neighbors of $v$ in $U$. 
For a positive constant $x$, let $\fraction{x}$ stand for $x - \lfloor x \rfloor$. 
We will use the following tools.
\begin{lemma}[Chernoff Bound]\label{bound:chernoff}
Let $X_1, \dots, X_n$ be i.i.d.~random variables such that $\Pr(X_i=1) = p$ and $\Pr(X_i=0)=1-p$ for each $i$. Then for any $t \geq 0$,
\begin{align*}
\Pr\left(\left|\sum_{i=1}^n X_i - np\right| > t\right) \leq 2e^{-t^2/(3np)}, \ \ & \text{ for } 0 \leq t \leq np, \\
\Pr\left(\left|\sum_{i=1}^n X_i - np\right| > t\right) \leq 2e^{-t/3}, \ \ & \text{ for } t > np.
\end{align*} 
\end{lemma}

\begin{lemma}[Lov\'asz Local Lemma]\label{LLL}\cite{EL, ASbook}
Let $\mathcal{E}_1, \dots, \mathcal{E}_n$ be $n$ events in any given probability space. Let $H$ be a simple graph with vertex set $[n]$ 
such that for each $i \in [n]$, the event $\mathcal{E}_i$ is mutually independent from the remaining events corresponding to non-neighbors of the vertex $i$, i.e., $\{ \mathcal{E}_j: j \neq i, \{i, j\} \notin E(H)\}$. Suppose there exist values $x_1, \dots, x_n \in (0,1)$ such that for each $i \in [n]$,
\[
\Pr(\mathcal{E}_i) \leq x_i \prod_{\{i,j\} \in E(H)} (1-x_j).
\]
Then the probability that none of the events $\mathcal{E}_i$ happens is positive, i.e., 
$
\Pr(\bigcap_{i=1}^n \bar{\mathcal{E}_i}) > 0.
$
\end{lemma}

\subsection{Random vertex partition through local lemma} 
Some part of our construction builds on ideas from~\cite{Pasy}.
In order to bypass the $\log n$ barrier for $d$ and be able to analyze the algorithm for all $1 \leq d \leq n$, we however need to phrase our construction differently, using quantization and the Lov\'asz Local Lemma (Lemma \ref{LLL}). 

The idea is to partition $V(G)$ into a big set $B = \{v_1, \dots, v_{|B|}\}$ and a small set $S$, where $|S| = (n/d)\cdot o(d)$. At the end, we will assure that 
$\fv(v_{i+1}) = \fv(v_i) + 1$ in $B$, and that vertices in $S$ have larger weights than those in $B$. 
Our argument 
divides into three steps.
Step 1 includes a random construction  
positioning weights in $B$ close to expected values, which are relatively sparsely distributed.
In Step 2 we modify 
the weights of edges across $B$ and $S$ to make vertices in $B$ have the desired weights. This is also 
the main purpose of singling out the set $S$. One benefit of $S$ being small compared to $B$ is that if we assign 
heavy weights between $S$ and $B$, then weights of vertices in $S$ are expected to increase more significantly than those in $B$. Step 3 is to modify weights in $S$ in order to make them all pairwise distinct. 


Fix parameters $\epsilon,\gamma$ such that $\epsilon \in (0,1/4)$ and $0< 2\gamma < \epsilon$.
Let $G$ be an $n$-vertex $d$-regular graph.
Set $\sstar =13\ceiling{d^{1/2+\epsilon}/13}$, note that $\sstar\in [d^{1/2+\epsilon},d^{1/2+\epsilon}+13)$ and $13|\sstar$. Unless specified, we always assume $d$ is sufficiently large in terms of $\gamma$. 

We first 
describe the main random ingredient of the construction. 
Let $X_v$ for $v \in V(G)$ be i.i.d.~ uniform random variables, $X_v\sim U[0,1]$. 
We use the values of $X_v$'s to separate the vertices into $d$ bins $B_i$ where 
$B_i = \{v\in V(G): (i-1)/d \leq X_v < i/d\}$ for $1 \leq i < d$ and $B_d = \{v\in V(G): 1-1/d \leq X_v \leq 1\}$; note that in expectation, each $B_i$ includes $n/d$ vertices.
Let the big set, consisting of most of the bins be defined as  $B=\bigcup_{1\leq i\leq d-\sstar}B_i$.
The remaining bins form a small set $S$, which we partition into regular $13$ subsets  
$S_i = \{ \bigcup B_j: d-(14-i){\sstar/13}  < j \leq d-(13-i){\sstar/13}\}$ for  $1 \leq i \leq 13$, 
hence $S=V(G)\setminus B = \bigcup_{1\leq i\leq 13}S_i$. 

Finally, we label some edges as ``corrected" to satisfy a subtle technical issue
(and guarantee later that the average weight of edges weighted $\integer{n/d}+1$ and $\integer{n/d}+2$ is exactly  $(n/d)+1$).
More precisely, we randomly label an edge with both end vertices in $B$ {\it ``corrected"}  independently with probability $\max(\fraction{n/d}, 1-\fraction{n/d})$, which is at least $1/2$.


%

\begin{lemma} \label{lem:main}
With positive probability, the following statements hold simultaneously if $d$ is large enough.
\begin{enumerate}
\item \label{item:vSi}($C_{vS_i}$) For each $v \in V(G)$ and $1 \leq i \leq 13$, $\deg_{S_i}(v) \in [\sstar/13 - d^{1/2+\gamma}, \sstar/13 + d^{1/2+\gamma}]$. 
\item\label{item:vS} ($C_{vS}$)  For each $v \in V(G)$, $\deg_S(v) \in [\sstar - 13d^{1/2+\gamma}, \sstar + 13d^{1/2+\gamma}]$, or equivalently, $\deg_B(v) \in [d-\sstar - 13d^{1/2+\gamma}, d-\sstar + 13d^{1/2+\gamma}]$.
\item\label{item:vB}($C_{vB}$) For each $v\in V(G)$, if $v \in B_i\cap B$ for some $i$, then the number of edges between $v$ and $\{\bigcup B_j, {d-s^*-i+1< j \leq d-s^*}\}$ is in the interval 
$[(i-1) - d^{1/2+\gamma}, (i-1) + d^{1/2+\gamma}]$.
\item \label{item:vBcorrect} ($C_{vB}'$) For each $v\in V(G)$, if $v \in B_i\cap B$ for some $i$, then the number edges between $v$ and  $\{\bigcup B_j, {d-s^*-i+1< j \leq d-s^*}\}$ that are labeled ``corrected" is in the interval $[(i-1)\alpha- \alpha d^{1/2+\gamma},  (i-1)\alpha + \alpha d^{1/2+\gamma} ]$ where $\alpha = \max(\fraction{n/d}, 1-\fraction{n/d})$. 
\item \label{item:Ci} ($C_{i}$) For each $1 \leq i \leq d$, 
$|\bigcup_{j \leq i} B_i| \in [ in/d - nd^{\gamma}/\sqrt{d},  in/d + nd^{\gamma}/\sqrt{d}]$.
\item\label{item:CSi} $(C_{S_i})$ For each $1 \leq i \leq 13$, $|S_i| \in [s^* n/(13d)- 2nd^{\gamma}/\sqrt{d}, s^* n/(13d) + 2nd^{\gamma}/\sqrt{d}]$.
\end{enumerate}
\end{lemma}
\begin{proof}
Let $\cE_{vS_i}$
be the bad event that $C_{vS_i}$ 
does not hold for given $v \in V(G)$, $1 \leq i \leq 13$. We analogously denote by 
$\cE_{vS}$, $\cE_{vB}, \cE_{vB}', 
\cE_i, 
\cE_{S_i}$ the remaining bad events. 
We first bound 
the probability of each of these, and then use Lov\'asz Local Lemma to show that  with positive probability none of these bad events happen. 

Fix $v\in V(G)$ and let  us consider 
$\cE_{vS_i}$ for any given $1 \leq i \leq 13$. 
Since each of $d$ neighbors 
of $v$ is independently included in 
$S_i$ 
with probability exactly $\sstar/(13d)$,
where $d^{1/2+\gamma}<\sstar/13 
<d$ for $d$ large enough, by the Chernoff Bound, 
\[ \Pr(\cE_{vS_i}) 
< 2e^{-d^{1+2\gamma}/(3d)} 
< 2e^{-d^{2\gamma}/6}. 
\]
As the events $C_{vS_i}$ imply $C_{vS}$,
we proceed to compute  the conditional probabilities $\Pr(\cE_{vB}|v \in B_i)$ and $\Pr(\cE_{vB}'|v \in B_i)$. These are trivially $0$ for $i = 1$. Thus we next assume $2\leq i \leq d-\sstar$. 
As each of $d$ neighbors $u$ of $v$ has probability $(i-1)/d$ 
to be in $\bigcup_{d-s^*-i+1< j \leq d-s^*}B_j$ and probability $\alpha(i-1)/d$ 
to be in $\bigcup_{d-s^*-i+1< j \leq d-s^*}B_j$ and simultaneously form a corrected edge $uv$, by the two Chernoff Bounds, since $\alpha\geq 1/2$,
\begin{align*}
\Pr(\cE_{vB} | v\in B_i) \leq & 2\exp\left({-\frac{d^{1+2\gamma}}{3
\max\left(
i-1,d^{1/2+\gamma}
\right) } }\right)
< 2e^{-d^{1+2\gamma}/(3d)} <2e^{-d^{2\gamma}/6}, \\
\Pr(\cE_{vB}' | v\in B_i) \leq & 2\exp\left({-\frac{\alpha^2d^{1+2\gamma}}{3
\max\left(
\alpha(i-1),\alpha d^{1/2+\gamma}
\right) } }\right)
< 2e^{-\alpha d^{1+2\gamma}/(3d)} \leq 2e^{-d^{2\gamma}/6}.
\end{align*}
Since by definition  $\Pr(\cE_{vB}|v \in S) = 0$ and $\Pr(\cE_{vB}'|v \in S) = 0$, thus by the law of total probability,
 $\Pr(\cE_{vB}) < 2e^{-d^{2\gamma}/6}$ and  $\Pr(\cE_{vB}') < 2e^{-d^{2\gamma}/6}$. 




To finally estimate $\Pr(\cE_i)$, 
we note that for any $i$ each of $n$ vertices 
is independently included in $\bigcup_{j \leq i}B_i$ 
with probability $i/d$. 
Thus by the Chernoff Bound,
\begin{align*}
\Pr(\cE_{i}) \leq & 2\exp\left({-\frac{n^2d^{2\gamma}/d}{3
\max\left(
in/d, nd^{\gamma}/\sqrt{d}
\right)} }\right)
\leq 2e^{-\frac{n^2d^{2\gamma-1}}{3n}} =  2e^{-n/(3d^{1-2\gamma})}.
\end{align*}

Since conditions (\ref{item:vSi}) and (\ref{item:Ci}) of the lemma imply conditions (\ref{item:vS}) and (\ref{item:CSi}), respectively, we just need to show that with positive probability none of  $\cE_{vS_i}, \cE_{vB}, \cE_{vB}', \cE_i$ 
holds.  
We will apply the  Lov\'asz Local Lemma (Lemma \ref{LLL}).
There are $13n$ events of type $\cE_{vS_i}$
(for each $v\in V(G)$ and $1 \leq i \leq 13$), $n$ events of type $\cE_{vB}$, 
$n$ events of type $\cE_{vB}'$ and $d$ events of type $\cE_i$.
%
Note 
 that for any given $v$ and $i$,
each of the events $\cE_{vS_i}, \cE_{vB},\cE_{vB}'$ is mutually independent of all other events $\cE_{uS_j},\cE_{uB},\cE_{uB}'$ with $u$ at distance at least $3$ from $v$ in $G$, i.e. all but most $13(d^2+1)+(d^2+1)+(d^2+1) < 16 d^2$ such events.
%
%
%
We assign value $x = d^{-2}/1600$ to each $\cE_{vS_i}, \cE_{vB},\cE_{vB}'$, and assign value $y = d^{-1}/100$ to all 
$\cE_i$. Therefore, in order to apply Lemma \ref{LLL} we just need to check that
 \[ \begin{cases}  2e^{-d^{2\gamma}/6}\leq   x (1-x)^{16d^2}(1-y)^{d}
 \\ 
  2e^{-n/(3d^{1-2\gamma})} \leq y (1-y)^d (1-x)^{15n}
  \end{cases}.\]
  Note that $1-a \geq e^{-10a}$ for $0 \leq a \leq 0.5$. Thus it is sufficient to show that: 
  \[ \begin{cases}  2e^{-d^{2\gamma}/6}\leq   e^{\ln(d^{-2}/1600)} e^{-160d^2 \cdot (d^{-2}/1600)} e^{-10 d \cdot d^{-1}/100}
 \\ 
  2e^{-n/(3d^{1-2\gamma})} \leq e^{\ln(d^{-1}/100)}    e^{-10\cdot d\cdot d^{-1}/100} e^{-150n \cdot d^{-2}/1600}, 
  \end{cases}\] 
which is equivalent to:
    \[ \begin{cases}  d^{2\gamma}/6\geq \ln 2  + \ln(1600d^{2}) + 1/10 + 1/10
 \\ 
  n/(3d^{1-2\gamma}) \geq 
 \ln 2+ \ln(100d)    + 1/10+ 15n/(160d^2) 
  \end{cases}.\]
These two inequalities above hold when $d$ is sufficiently large in terms of $\gamma$. The thesis thus follows by Lemma \ref{LLL}. 
\end{proof}

\subsection{Assigning weights}
Suppose all statements in Lemma \ref{lem:main} hold.
We will assign and modify edge weights in $G$ in three steps.
Whenever needed we assume $d$ is large enough in terms of $\gamma$.
\paragraph{Step 1.} The purpose of 
this step is to construct an initial weighting function $f_1: E(G) \to \mathbb{N}$ so that 
all $v \in B_i$ have weights 
very close to $(n/d)i$
 for each $i$. For this aim 
 for every edge $\{u,v\}$ with
$v \in B_i \cap B$ and $u \in B_j \cap B$ we define $f_1(\{u,v\})$ as follows.
If $\fraction{n/d} \geq 1/2$ and 
$d-\sstar -i+1 < j \leq d-\sstar$, let $f_1(\{u,v\}) = \integer{n/d}+1$ if $\{u,v\}$ is not a corrected edge, and let $f_1(\{u,v\}) = \integer{n/d}+2$ if it is a corrected edge. 
If $\fraction{n/d} < 1/2$ and $d-\sstar -i+1 < j \leq d-\sstar$, let $f_1(\{u,v\}) = \integer{n/d}+2$ if $\{u,v\}$ is not a corrected edge, and let $f_1(\{u,v\}) = \integer{n/d}+1$ if it is a corrected edge. 
We next define $\omega =\max( \ceiling{n/d^{1+\epsilon - 2\gamma}},2)$ and set $f_1(e)=i\omega + \ceiling{n/d}$ for every edge $e$ across $B$ and $S_i$ for $1\leq i\leq 13$.
We finally set $f_1(e')=1$ for all the remaining edges $e'\in E(G)$.

%
Consider any $v \in B_i \cap B$. We assume  $\fraction{n/d} \geq 1/2$, as the analysis and result in the opposite case is essentially   the same.
By the definition of $f_1$ and Lemma \ref{lem:main}, 
since $1< n/d$ and $\omega\leq \ceiling{n/d}< 2n/d$,
\begin{align*}
\fvv{1}(v) = &  \sum_{u:\{u,v\}\in E(G),u\in S} f_1(\{u,v\}) + \sum_{u:\{u,v\}\in E(G),u\in B} f_1(\{u,v\}) \geq  \left(\sum_{j=1}^{13}  (j\omega + \ceiling{n/d}) (\sstar/13 - d^{1/2+\gamma})\right)\\
  & + \left(\integer{n/d}((i-1)-d^{1/2+\gamma}) + \fraction{n/d}((i-1)-d^{1/2+\gamma}) + \left(d-\sstar-13d^{1/2+\gamma}\right)\right)
\\
> & \left(\left( 7\omega + \ceiling{n/d}  \right)\sstar - d^{1/2+\gamma} (91\omega + 13 \ceiling{n/d})\right) 
+ \left((n/d)(i-1)+d-\sstar-16n/d^{1/2-\gamma} \right)
\\
\geq & \left((n/d)(i-1)+d+  \left( 7\omega + \ceiling{n/d}-1  \right)\sstar\right) -  224n/d^{1/2-\gamma}.
\end{align*}
%
By almost the same reasoning we may obtain an analogous upper bound for $\fvv{1}(v)$, implying that
\begin{align}
 \left|\fvv{1}(v) -\left((n/d)(i-1)+d+  \left( 7\omega + \ceiling{n/d}-1  \right)\sstar\right)\right| \leq 
  224n/d^{1/2-\gamma}. \label{eq:fvcenter} 
\end{align}
Moreover, the following claim holds.
\begin{claim}\label{claim:f1}
For any edge $e \in E(B)$, $f_1(e) \in [1, \integer{n/d}+2]$. For any edge $e$ across $B$ and $S$, $f_1(e)\in [\ceiling{n/d}, \ceiling{n/d}+  13\omega ]$. For any edge $e \in E(S)$, $f_1(e)=1$.  
\end{claim}

\paragraph{Step 2.} Consider a linear ordering $v_1,v_2,\ldots$ of the vertices in $B$ such that
$X_{v_j} \geq X_{v_i}$ if $j \geq i$ (where $X_v$'s refer to values of the random variables used within the proof of Lemma~\ref{lem:main} for which all conditions of the lemma hold; we may assume these are all distinct, as this is true with probability $1$). 
To adjust edge and vertex weights we will provide $f_2: E(G)\to \mathbb{N}$ supported on edges between $B$ and $S$ (i.e. equal to $0$ for the remaining edges) so that as a result, for $f_{12}:=f_1+f_2$, the following conditions hold:
(1) each $v_k \in B$ has weight 
$k+d+  \left( 7\omega + \ceiling{n/d}-1  \right)\sstar +  \ceiling{250n/d^{1/2-\gamma}}$;
(2)  for any $u \in S, v\in B$, $\fvv{12}(u)-\fvv{12}(v) >0$; 
and finally: (3) for $u \in S_{i+1}, v\in S_i$ with $1 \leq i \leq 12$, $\fvv{12}(u)-\fvv{12}(v)$ is large enough to provide a buffer for weight adjustments in Step 3.
%

Suppose $v_k \in B_i \cap B$. Then $ |\bigcup_{j\leq i-1} B_j| \leq k \leq |\bigcup_{j\leq i} B_j|$,
and thus, by Lemma \ref{lem:main} (\ref{item:Ci}), $(i-1)n/d - n/d^{1/2-\gamma} \leq k \leq in/d +   n/d^{1/2-\gamma}$. Therefore,
\begin{align}
\left| k -(i-1)n/d\right| \leq 2 n/d^{1/2-\gamma}.  \label{eq:k}
\end{align}
By (\ref{eq:fvcenter}), (\ref{eq:k}) and the triangle inequality,
\begin{align}
|\fvv{1}(v_k) -\left(k+d+  \left( 7\omega + \ceiling{n/d}-1  \right)\sstar\right)| \leq 
& 226n/d^{1/2-\gamma}. 
\label{eq:fvsurplusmain}
\end{align}


\begin{claim}\label{claim:vidistinct}
There exists 
 $f_2:E(G)\to \mathbb{N}$ supported on edges across $B$ and $S$ such that 
$\| f_2 \|_\infty \leq \ceiling{10^3n/d^{1+\epsilon-\gamma}}$ and 
for each $v_k \in B$, $\fvv{12}(v_k) =
k+d+  \left( 7\omega + \ceiling{n/d}-1  \right)\sstar +  \ceiling{250n/d^{1/2-\gamma}}$, provided 
$d$ is sufficiently large in terms of $\epsilon, \gamma$.
\end{claim}
\begin{proof}
Note that by~(\ref{eq:fvsurplusmain}), the weight 
of every $v_k \in B$ is smaller than the target value 
$k+d+  \left( 7\omega + \ceiling{n/d}-1  \right)\sstar +  \ceiling{250n/d^{1/2-\gamma}}$, 
while we need to add no more than $500n/d^{1/2-\gamma}$ to achieve it.
This discrepancy can be leveled up by adding appropriate quantities to weights of edges between $v_k$ and $S$, thereby defining $f_2$.
As by Lemma \ref{lem:main} (\ref{item:vS}), $d_S(v_k)\geq \sstar - 13d^{1/2+\gamma} > \sstar/2$, it is sufficient to add to every edge weight between $v_k$ and $S$ at most
$\ceiling{(500n/d^{1/2-\gamma})/(\sstar/2)} \leq \ceiling{10^3n/d^{1+\epsilon-\gamma}}$.
%
%
%
\end{proof}

\begin{claim}\label{claim:buffer}
 For every $u \in S$ and $v\in B$, $\fvv{12}(u) > \fvv{12}(v)$. 
For each $2\leq i \leq 13$ and every $u \in S_{i}$, $u' \in S_{i-1}$, we have $\fvv{12}(u) - \fvv{12}(u') \geq  0.4\omega d$.  
\end{claim}
\begin{proof}
Consider $u\in S_i$ for any fixed $1\leq i \leq 13$. By Claim \ref{claim:vidistinct}, $f_{12}(e)\geq f_1(e)$ for every edge $e$. Hence, as due to Lemma \ref{lem:main} (\ref{item:vS}), $\deg_B(u) \geq d-\sstar - 13d^{1/2+\gamma} \geq d - 2\sstar$, by the definition of $f_1$ and the fact that $\omega\leq \ceiling{n/d}$, 
\begin{align}
\fvv{12}(u)  \geq &  \left( i\omega+\ceiling{n/d} \right) (d-2\sstar) 
\geq i\omega d + \ceiling{n/d}d 
- 2\sstar (i+1) \ceiling{n/d}. 
\label{eq:fuup1}
\end{align}
By Claim \ref{claim:vidistinct}, for every $v \in B$, $\fvv{12}(v) \leq \fvv{12}(v_{|B|}) = |B| +d+  \left( 7\omega + \ceiling{n/d}-1  \right)\sstar +  \ceiling{250n/d^{1/2-\gamma}} < n + d+ 9\ceiling{n/d}\sstar$. 
Thus for $u \in S_i$ and $v \in B$, as $1\leq i\leq 13$ and $d\leq 0.5\omega d$, together with (\ref{eq:fuup1}),
\begin{align*}
 \fvv{12}(u)- \fvv{12}(v) 
\geq &  i\omega d - 2\sstar (i+1) \ceiling{n/d} -d- 9\ceiling{n/d}\sstar 
\geq 0.5\omega d - 37\ceiling{n/d}\sstar > 0, 
\end{align*} 
where the last inequality holds, as 
 $\omega d \gg (n/d) \sstar$ when $d \to \infty$. The first claim is thus proved. 

Consider now any $u'\in S_{i-1}$ for a given  $2\leq i \leq 13$.  
By Claim \ref{claim:vidistinct} and the definition of $f_1$, $f_{12}(e)\leq (i-1)\omega+\ceiling{n/d} + 10^3n/d^{1+\epsilon - \gamma}+ 1$ for every edge $e$ incident with $u'$. Thus
\begin{align}
 \fvv{12}(u')  \leq  &
\left( (i-1)\omega+\ceiling{n/d} + 10^3n/d^{1+\epsilon - \gamma}+ 1\right) d 
=  i\omega d + \ceiling{n/d}d +d - \omega d +  10^3n/d^{\epsilon - \gamma}. 
\label{eq:fulow1}
\end{align}
Thus combing (\ref{eq:fuup1}) and (\ref{eq:fulow1}), for $u' \in S_{i-1}$ and $u \in S_i$,  
as  $\omega d \gg n \sstar /d, \omega d \gg n/d^{\epsilon-\gamma}$ when $d \to \infty$,
\begin{align*}
& \fvv{12}(u) - \fvv{12}(u') \geq  & \omega d  -d -10^3n/d^{\epsilon - \gamma}  - 2\sstar (i+1) \ceiling{n/d}
\geq 0.5\omega d -10^3n/d^{\epsilon - \gamma}  - 28\ceiling{n/d}\sstar 
\geq 0.4\omega d. 
\end{align*}
\end{proof}
By Claims \ref{claim:f1} and \ref{claim:vidistinct}, as a summary, the following holds after Step 2. 
\begin{claim}\label{claim:step2weights}
For any edge $e \in E(B)$, $f_{12}(e) \in [1, \integer{n/d}+2]$. For any edge $e$ across $B$ and $S$, $f_{12}(e)\in [\ceiling{n/d}, \ceiling{n/d}+  13\omega  + \ceiling{10^3n/d^{1+\epsilon-\gamma}}]$. For any edge $e \in E(S)$, $f_{12}(e)=1$. 
\end{claim}

\paragraph{Step 3.} 
In this step, we introduce $f_3: E(G) \to \mathbb{N}$ 
that is only supported on $E(S)$ such that all vertices in $S$ have distinct weights with respect to $f=f_1+f_2 + f_3$. We will moreover show that Claim \ref{claim:buffer} implies that $f$ attributes distinct weights to all vertices in $G$.
For this aim we will adapt the algorithm from
~\cite{Pasy}, which was modeled on the idea of 
Kalkowski, Karo\'nski, and Pfender~\cite{KKP}. 

 
Let $\mathsf{AP}$ be the family of  sets of the following form $\sset=\{ (2\lambda) \ceiling{n/(3d)}+a, (2\lambda+1) \ceiling{n/(3d)}+a\}$ where $\lambda, a$ are integers with $\lambda \geq 0$, $a\in [0, \ceiling{n/(3d)}-1]$. 
Note the sets in $\mathsf{AP}$ with all possible values of $\lambda, a$ partition the non-negative integers, where different values of $a$ correspond to  $\ceiling{n/(3d)}$ different congruence classes, denoted by $\mathcal{C}_a = \{a+ k\ceiling{n/(3d)}, k \in \mathbb{N}\}$. 
Our primary goal is to attribute the weight of every vertex $v\in S$ to appropriately chosen $\sset_{v} \in \mathsf{AP}$ 
so that for each  $1 \leq l \leq 13$, vertices in $S_l$  have associated pairwise distinct $\sset_{v}$'s, which are thus disjoint.

We initialize $f_3$ by setting $f_3(e) = \ceiling{n/(3d)}$ for all $e\in E(S)$. 
Given an ordering $v_1, \dots, v_{|S|}$ of vertices in $S$ (specified later), each edge $\{v_i, v_j\}$ with $i<j$ is called a {\it forward edge} of $v_i$ and a  {\it backward edge} of $v_j$. The algorithm will sequentially processes $v_i$'s modifying $f_3$ on edges in $S$ incident to currently analyzed $v_i$. For  $v_1$ no modifications are needed -- we simply let $\sset_{v_1}$ be the set in $\mathsf{AP}$ that contains the current value of $\fv(v_1)$ and move on (to $v_2$). 
Then for every consecutive $i\geq 2$, we will choose a special set 
${\sset_{v_i}} \in \mathsf{AP}$ and guarantee that $\fv(v_i)$ belongs in $\sset_{v_i}$ since the end of step $i$ till the end of entire algorithm.
We admit 
two options to modify $f_3$ on backward edges of the given $v_i$:  either by adding $0$ or one of the values in $\{\pm \ceiling{n/(3d)}\}$. Specifically, say $\{v_i,u\}$ is a backward edge of $v_i$. If the current value of $\fv(u)$ is the smaller value in $\sset_u$, we admit adding $0$ or $\ceiling{n/(3d)}$ to $f_3(\{v_i,u\})$; if $\fv(u)$ is the larger value in $\sset_u$, we 
in turn admit subtracting $0$ or $\ceiling{n/(3d)}$ from  $f_3(\{v_i,u\})$.
Thereby the updated $\fv(u)$ will always remain in $\sset_u$, as desired. 
We finally admit adding any value in $\{0, 1, \dots,\ceiling{3n/d}\}$ to the weights $f_3(e)$ of all forward edges of $v_i$, which will in particular allow us to determine the congruence class $\fv(v_i)$ will eventually land in. 


We now specify the ordering $v_1, v_2, \dots$  of the vertices in $S$. At the beginning we arrange the vertices in $\bigcup_{1\leq i \leq 12}S_i$ according to the values of $X_{v_i}$, from the smallest to the largest, and thus consistently with the order of $S_1, \dots, S_{12}$. 
The last in the ordering are vertices from $S_{13}$, which are ordered differently due to some technical subtlety concerning vertices without forward edges.
Suppose $C_1, \dots, C_K$ are the connected components in $S_{13}$, ordered arbitrarily. Each component has at least two vertices by Lemma \ref{lem:main}(\ref{item:vSi}). For each $C_i$, we use reversed BFS to order its vertices and denote $r_i, t_i$ the last two vertices in $C_i$. (Thus $t_i$ is the root of the tree in BFS; $\{r_i, t_i\}\in E(S)$.) Let $R = \{r_1, \dots, r_K\}$ and $T = \{t_1, \dots, t_K\}$. We finally define the ordering in 
$S_{13}$ by concatenating the orderings of $C_1,\ldots,C_K$. Note that by Lemma \ref{lem:main}(\ref{item:vSi}), the set of {\it terminal vertices}, i.e.,  vertices with no forward edges in the obtained ordering in $S$, is $T$.

We now show specific procedures which will allows us to achieve the desired goal. Suppose we are in step $i$, i.e. we are analyzing $v_i \in S_t$, where $1 \leq t \leq 13$, and that $v_i \notin R\cup T$, hence $v_i$ has at least one forward edge, say $e_i$. 
The existing sets $\sset_{u}$ for $u$ prior to $v_i$ in $S_t$ correspond to at most $|S_t|$ congruence classes with possible duplicates. Therefore, there must be a congruence class $\mathcal{C}_a$ that includes at most
$|S_t|/\ceiling{n/(3d)}$ prior sets $\sset_u$ with $u \in S_t$.
Thus we may include the weight of $\fv(v_i)$ in $\mathcal{C}_a$ by adding one of admissible values in $\{0, 1, \dots,\ceiling{3n/d}\}$ to the weight $f_3(e_i)$. 
We then modify the rest of the forward edges of $v_i$ by adding $0$ or $\ceiling{n/(3d)}$ and 
 change the weights of some backward edges of $v_i$ by $\ceiling{n/(3d)}$ according to the specified rules, if necessary. 
Note that this way we may obtain $\deg_S(v_i)$ 
consecutive terms in $\mathcal{C}_a$ as weights of $v_i$. 
Since each prior set $\sset_u$ blocks at most two consecutive terms in $\mathcal{C}_a$, 
we can find
this way an attainable $\fv(v_i)\in\mathcal{C}_a$ which is not blocked if only 
 $\deg_S(v_i) > 2|S_t|/\ceiling{n/(3d)}$.
This is however implied by an even stronger inequality, which 
holds by Lemma \ref{lem:main} (\ref{item:CSi})(\ref{item:vS}):
\begin{align}
 4|S_t|/\ceiling{n/(3d)}+2 \leq & 4(\sstar n/(13d)+2n/d^{1/2-\gamma})/(n/(3d)) +2 \leq 12\sstar /13 + 24d^{1/2+\gamma}+2 
\nonumber\\
< & \sstar - 13d^{1/2+\gamma} \leq \deg_S(v_i).\label{eq:choice}
\end{align}
We finally set $\sset_{v_i}$ as the only set in $\mathsf{AP}$ containing the attained weight of $v_i$. 

We are left to show how to handle $r_j, t_j \in R \cup T$, where $\{r_j,t_j\}$ is the only forward edge of $r_j$. 
We analyze both vertices simultaneously in a similar manner as above. 
Recall $r_j, t_j \in S_{13}$. 
First,
by an averaging argument,  we can choose an admissible addition to $f_3(\{r_j, t_j\})$ such that the two new congruence classes of $\fv(r_j), \fv(t_j)$ each includes 
at most $2|S_{13}|/\ceiling{n/(3d)}$ prior sets $\sset_u$ with $u\in S_{13}$, disregarding temporarily  $\sset_{r_j}$ from the point of view of $t_j$. 
Next, analogously as above,  by~(\ref{eq:choice}), we can change 
the weights of backward edges of $r_j$ 
by $\pm\ceiling{n/(3d)}$ 
so that the resulting $\fv(r_j)$ belongs to $\sset_{r_j}\in \mathsf{AP}$  
disjoint from those of the prior vertices in $S_{13}$. 
Finally, we analogously adjust the weights of all backward edges of $t_j$ except $\{r_j, t_j\}$ 
so that the resulting $\fv(t_j)$ belongs to $\sset_{t_j}\in \mathsf{AP}$  
disjoint from those of the prior vertices in $S_{13}$ including $\sset_{r_j}$, which is again feasible by (\ref{eq:choice}) (where ``+2'' was incorporated in this inequality to facilitate distinguishing $\sset_{t_j}$ from $\sset_{r_j}$). 


\begin{claim}\label{claim:maxweight}
For every edge $e$ of $G$, $1\leq f(e)\leq \ceiling{n/d}+  13\omega  + \ceiling{10^3n/d^{1+\epsilon-\gamma}}$.
\end{claim}
\begin{proof}
By Claim \ref{claim:step2weights} all edge weights were in the interval $[1, \ceiling{n/d}+  13\omega  + \ceiling{10^3n/d^{1+\epsilon-\gamma}}]$ prior to Step 3, while edges in $E(S)$ were assigned $1$.
Within Step 3 only edges in $E(S)$ could have been modified, each at most twice
(once as a forward edge and once as a backward edge). Thus  $f_3(e) \in [0, 3\ceiling{n/(3d)}]$, and hence the result follows. 
%
\end{proof}

\subsection{
Proof of Theorems \ref{thm:main2} and \ref{thm:main1}} 

Note that by the algorithm applied above, $\fv(v)$'s are pairwise distinct for vertices in the same $S_i$ for $1\leq i \leq 13$. 
We first show that if $u \in S_{i}$ and $u' \in S_{i-1}$, where $2\leq i \leq 13$, then $\fv(u) > \fv(u')$. By Claim \ref{claim:buffer},  $\fvv{12}(u) - \fvv{12}(u') \geq 0.4 \omega d$. 
Moreover, by the algorithm in Step 3, 
$\fvv{12}(u'')\leq \fvv{}(u'') \leq \fvv{12}(u'')+3\ceiling{n/(3d)} \deg_S(u'')$ for every $u''\in S$, as
$0 \leq f_3(e)\leq 3\ceiling{n/(3d)}$ for each $e\in E(S)$. Hence, by Lemma \ref{lem:main} (\ref{item:vS}),
since $\omega d \gg n\sstar/d$ as $d \to \infty$,
\begin{align*}
\fv(u) - \fv(u') \geq &
0.4 \omega d - 3\ceiling{n/(3d)} \deg_S(u') \geq 0.4 \omega d - 4(n/d)2\sstar > 0. 
\end{align*} 
Thus all vertices in $S$ have pairwise distinct weights. 
For any vertices $u\in S$ and $v\in B$, since $\fvv{3}(v)=0$ and $\fvv{3}(u) \geq 0$, by 
Claim \ref{claim:buffer}, 
$\fv(u) - \fv(v) > 0$. 
Finally, as by Claim~\ref{claim:vidistinct}, the weights of the vertices in $B$ formed a $|B|$-element segment of integers after Step 2 and have not changed ever since, all vertices in $G$ have distinct weights.

Suppose $d^{1+\epsilon-2 \gamma} \geq n$. Then $\omega=2$ and $\ceiling{10^3n/d^{1+\epsilon-\gamma}}=1$ for $d$ large enough. Thus  by Claim \ref{claim:maxweight}, there is $d_0$ such that $\| f\|_\infty \leq \ceiling{n/d}+27<n/d+28$ 
for $d\geq d_0$. 
As there are only finitely many graphs with  $d< d_0$ and $d^{1+\epsilon -2 \gamma} \geq n$ (while e.g. by~\cite{KKP}, $s(G)\leq 6 \ceiling{n/d}$),
Theorem \ref{thm:main1} and Corollary \ref{cor:main1} follow  due to 
taking $\beta = \epsilon - 2\gamma$, as 
$\epsilon\in (0,1/4)$ while $\gamma$ can be chosen arbitrarily small.

On the other hand, by Claim \ref{claim:maxweight},  $f$ is upper bounded by 
$\ceiling{n/d}+  13\omega  + \ceiling{10^3n/d^{1+\epsilon-\gamma}} 
< (n/d+1) +13(n/d^{1+\epsilon-2\gamma}+2)+ (n/d^{1+\epsilon-2\gamma}+1) = n/d(1+14/d^{\epsilon-2\gamma})+28$
 when $d$ is sufficiently large, say $d\geq d_0$ (where $d_0$ is a constant dependent on $\epsilon,\gamma$), hence Theorem~\ref{thm:main2} follows by taking $\beta = \epsilon - 2\gamma$ analogously as above. For $d< d_0$, we may finally again use the result in~\cite{KKP} 
implying that 
$s(G) < (n/d)(1+ 5) + 6$, where $5\leq C/d^\beta$ for large enough (in terms of $d_0$) constant $C$. Thus 
Corollary \ref{cor:main2} is proved as well. 

\section{Conclusion and remarks}
In this note, we proved a uniform upper bound $s(G) \leq \frac{n}{d}(1+C/d^{\beta}) + 28$, which 
confirms the Faudree-Lehel Conjecture for $d \geq n^{\alpha}$ 
for any fixed $\alpha>0.8$.
Our primary goal was to present a relatively short proof, hence we did not strive to optimize all constants and auxiliary functions within our argument. In particular, using a slightly more detailed analysis concerning $r_i,t_i$ in the algorithm in Step 3 (applied already e.g. in~\cite{KKP,MP}) and a few other minor alterations, one may easily reduce the constant $28$  to $16$ in all our main results (and the constant $14$ to $8$).

\paragraph{Acknowledgment.} 
The authors would like to thank Jacob Fox for helpful discussions.

\bibliographystyle{plain}
\bibliography{sample3}

\end{document}